\documentclass[a4paper]{article}

\usepackage{amsfonts}
\usepackage{graphicx}
\usepackage{amsthm}
\usepackage{amsmath}
\usepackage{amssymb}
\usepackage{enumerate}
\allowdisplaybreaks

\newtheorem{proposition}{Proposition}[section]
\newtheorem{lemma}[proposition]{Lemma}

\newtheorem{theorem}[proposition]{Theorem}
\newtheorem{conjecture}[proposition]{Conjecture}
\theoremstyle{definition}

\newtheorem{remark}[proposition]{Remark}
\numberwithin{equation}{section}
\newtheorem*{example*}{Example}

\begin{document}

\begin{center}
\LARGE
\textbf{Writing finite simple groups of Lie type \\ as products of subset conjugates}
\bigskip\bigskip

\large
Daniele Dona
\bigskip

\normalsize
HUN-REN Alfr\'ed R\'enyi Institute of Mathematics

Re\'altanoda utca 13-15, Budapest 1053, Hungary

\texttt{dona@renyi.hu}
\bigskip\medskip
\end{center}

\begin{minipage}{110mm}
\small
\textbf{Abstract.} The Liebeck-Nikolov-Shalev conjecture \cite{LNS12} asserts that, for any finite simple non-abelian group $G$ and any set $A\subseteq G$ with $|A|\geq 2$, $G$ is the product of at most $N\frac{\log|G|}{\log|A|}$ conjugates of $A$, for some absolute constant $N$.

For $G$ of Lie type, we prove that for any $\varepsilon>0$ there is some $N_{\varepsilon}$ for which $G$ is the product of at most $N_{\varepsilon}\left(\frac{\log|G|}{\log|A|}\right)^{1+\varepsilon}$ conjugates of either $A$ or $A^{-1}$. For symmetric sets, this improves on results of Liebeck, Nikolov, and Shalev \cite{LNS12} and Gill, Pyber, Short, and Szab\'o \cite{GPSS13}.

During the preparation of this paper, the proof of the Liebeck-Nikolov-Shalev conjecture was completed by Lifshitz \cite{Lif24}. Both papers use \cite{GLPS24} as a starting point. Lifshitz's argument uses heavy machinery from representation theory to complete the conjecture, whereas this paper achieves a more modest result by rather elementary combinatorial arguments.
\medskip

\textbf{Keywords.} Growth, conjugacy, finite simple groups.
\medskip

\textbf{MSC2020.} 20D06, 20F69, 20G40.
\end{minipage}
\medskip

\section{Introduction}

An important general question at the crossroad between finite group theory and combinatorics is whether we can represent a group $G$ as the product of few copies of a subset $A$ of $G$. The topic goes back at least to the first formulation of Babai's conjecture \cite{BS88}.

When we allow ourselves to take the product of conjugates of $A$, rather than of copies of $A$ itself, the number of necessary conjugates was generally expected to be very small, essentially the smallest possible up to multiplicative constant. This is the content of the conjecture below, due to Liebeck, Nikolov, and Shalev \cite{LNS12}.

\begin{conjecture}\label{co:conjlns}
There exists $N$ such that, for all finite simple groups $G$ and all subsets $A\subseteq G$ with $|A|\geq 2$, $G$ is the product of at most $N\frac{\log|G|}{\log|A|}$ conjugates of $A$.
\end{conjecture}

The conjecture generalizes a slightly earlier one by the same authors, where $A$ is a subgroup of $G$ rather than a subset \cite{LNS10}.

Many partial results have been known for years. The following two are relevant for us, in that they hold for all $G$ of Lie type and all $A$: the number of conjugates inside the product can be bounded by
\begin{align}
\left(\frac{\log|G|}{\log|A|}\right)^{N(r)} & \text{\ \ \ \ \ by \cite[Thm.~1]{LNS12},} \label{eq:boundlns} \\
N(r)\frac{\log|G|}{\log|A|} & \text{\ \ \ \ \ by \cite[Thm.~1.3]{GPSS13},} \label{eq:boundgpss}
\end{align}
where in both cases $N(r)$ is a function of the rank $r$ of $G$.

In an independent and far-reaching work that appeared during the preparation of this paper, Lifshitz \cite{Lif24} managed to complete the proof of Conjecture~\ref{co:conjlns}. For the Lie type case, Lifshitz's paper uses heavy machinery from character theory, stemming from Deligne-Lusztig theory (in the form of a recent character bound from Larsen and Tiep \cite{LT24}) and from the generalized Frobenius formula of \cite{GLPS24}.

Our main result is more modest in several ways than Conjecture~\ref{co:conjlns}. However, the proof relies on combinatorial arguments that are surprisingly quite elementary, much more than in \cite{Lif24}, and it may be of independent interest to see how far one can go with elementary tools.

\begin{theorem}\label{th:main}
For any $0<\varepsilon\leq\frac{1}{2}$ there exists $N_{\varepsilon}$ such that the following holds. Let $G$ be a finite simple group of Lie type, and let $A\subseteq G$ with $|A|\geq 2$. Then $G$ is the product of at most
\begin{equation*}
N_{\varepsilon}\left(\frac{\log|G|}{\log|A|}\right)^{1+\varepsilon}
\end{equation*}
conjugates of either $A$ or $A^{-1}$.
\end{theorem}

When $A$ is symmetric (for instance when $A$ is a subgroup), the result improves considerably on \eqref{eq:boundlns}. As observed by L.~Pyber (personal communication), Theorem~\ref{th:main} can be reformulated to directly improve \eqref{eq:boundgpss} as well: see Remark~\ref{re:pyber} for proof and comparisons.
\medskip

\textit{Note on chronology.} The result of this paper was already privately proved in August 2023, before the completion of the stronger work of Lifshitz \cite{Lif24}, and it was made privately available to L.~Pyber in January 2024. This paper ended up appearing publicly only after \cite{Lif24}. Both \cite{Lif24} and this paper depend on \cite{GLPS24}, in the latter case through Theorem~\ref{th:ptc}. No wrongful conduct can be ascribed to N.~Lifshitz, who was not yet a coauthor of \cite{GLPS24} in January 2024 and was unaware of the author's work until after the appearance of \cite{Lif24} in August 2024.

\section{Strategy}

Let us give here a general overview of our strategy. Throughout the paper, $r$ is the (untwisted) rank of $G$ and $q$ is the size of the base field of $G$. We want the set $A$ to ``grow quickly'' all the way to $G$, where with ``growing quickly'' we mean that we can find a product $X$ of $N$ conjugates of either $A$ or $A^{-1}$ where $|X|$ is particularly large; the optimal growth would be $|A|^{N}$ by order considerations, and we aim for not much less than that. With some caution, we are allowed to grow in several spurts: we can use one technique or result to make $A$ grow up to reaching a product $A'$ larger than some threshold, then make $A'$ grow using a different technique or result until we reach a product $A''$ larger than some second threshold, and so on.

Until $|A|$ is around $q^{r}$ we know that we grow very quickly (Proposition~\ref{pr:smallgrow}). A version of the product theorem valid up to conjugation (Theorem~\ref{th:ptc}) can make us grow for short intervals, such as from $q^{r}$ to $q^{Cr}$ and from $q^{r^{2}/C}$ to $|G|$ for any fixed constant $C$. The name ``product theorem'' comes from analogous results in the context of Babai's conjecture, such as the ones in \cite{Hel08} \cite{BGT11} \cite{PS16} \cite{BDH21}; such results already appear in the literature related to Conjecture~\ref{co:conjlns}, as they are used in the proof of \eqref{eq:boundlns}.

It remains to deal with growth from $q^{Cr}$ to $q^{r^{2}/C}$, which forms the main and novel part of our proof. Since the number of conjugacy classes is only $q^{4r}$ (Proposition~\ref{pr:fewclass}), in this interval we can assume to be working with subsets of conjugacy classes without losing too much of the size of $A$.

Two cases arise at this point: either our set $A$ (or rather $A^{-1}A$ or $AA^{-1}$, for technical reasons) occupies a large chunk of a conjugacy class, or it does not. In the former case, since we already know that classes grow quickly (Theorem~\ref{th:dmp}), we can make the large subset inherit the same growth rate: this is the content of \S\ref{se:xcldelta}, and the growth we achieve here is compatible not only with Theorem~\ref{th:main} but even with Conjecture~\ref{co:conjlns}. In the latter case, when $A$ occupies a small chunk of every class, there is enough manoeuvring space to produce a conjugate of $A$ completely disjoint from $A$; with such a conjugate in hand, we grow by standard arguments. In \S\ref{se:main} we tackle this second case and prove the main theorem.

We mention that many of the techniques, taken one by one, are applicable to more general groups. The pecularity of finite simple groups of Lie type is that the various pieces fit together to form a complete proof. To make an example, for $G=\mathrm{Alt}(n)$ there is a much wider gap between the size of the smallest conjugacy class and the number of conjugacy classes, in the ballpark of $n^{2}$ and $e^{\sqrt{n}}$ respectively, whereas for $G$ of Lie type we have $q^{r}$ and $q^{4r}$; therefore, a result like Theorem~\ref{th:ptc} would not be enough in $\mathrm{Alt}(n)$ to bridge the gap between Proposition~\ref{pr:smallgrow} and the main part of our strategy.

\section{Tools}

We start by quoting some group-theoretic results about finite simple groups of Lie type. First we have a lower bound on the size of any nontrivial conjugacy class: the following is \cite[Prop.~2.3]{GPSS13}, which relies on \cite[\S 5]{KL90}.

\begin{proposition}\label{pr:qnclass}
Let $G$ be a finite simple group of Lie type of rank $r$ over $\mathbb{F}_{q}$. Then $|G|\leq q^{8r^{2}}$, and every nontrivial conjugacy class $C$ has size $|C|\geq q^{r}$.
\end{proposition}

We also need an upper bound on the number of conjugacy classes. The following is \cite[Thm.~1]{LP97}.

\begin{proposition}\label{pr:fewclass}
Let $G$ be a finite simple group of Lie type of rank $r$ over $\mathbb{F}_{q}$. There are $\leq(6q)^{r}\leq q^{4r}$ conjugacy classes in $G$.
\end{proposition}

Now we list several results more specifically about growth in $G$. First, we do not have to worry about the case of bounded rank thanks to the result below, which comes from \cite[Thm.~1.3]{GPSS13}.

\begin{theorem}\label{th:boundr}
Fix $r_{0}$, and let $G$ be a finite simple group of Lie type of rank $\leq r_{0}$. Conjecture~\ref{co:conjlns} holds for this $G$, with $N$ depending on $r_{0}$.
\end{theorem}

The growth is already known to be as fast as possible for small sets, at least until we reach a size comparable to the smallest conjugacy class, as shown below.

\begin{proposition}\label{pr:smallgrow}
Let $G$ be a finite simple group of Lie type of rank $r$ over $\mathbb{F}_{q}$, and let $S\subseteq G$ with $|S|\geq 2$. There are some positive integer $m$ and some set $X$ with
\begin{align*}
X & =S^{h_{1}}S^{h_{2}}\ldots S^{h_{m}} \ \ (h_{i}\in G), & |X| & =|S|^{m}\geq\frac{q^{\frac{1}{2}r}}{|S|}.
\end{align*}
\end{proposition}

\begin{proof}
Combine \cite[Lemma 3.1]{GPSS13} and Proposition~\ref{pr:qnclass}.
\end{proof}

For $G$ of Lie type and $S\subseteq G$, we can also find a product $SS^{g}$ growing exponentially in size. This is \cite[Thm.~1.6]{GLPS24}, improving on \cite[Thm.~1.4]{GPSS13}, which as we mentioned is analogous to results called ``product theorems'' in the context of Babai's conjecture.

\begin{theorem}\label{th:ptc}
There are absolute constants $\tau>0$ and $b$ such that, for any finite simple non-abelian group $G$ and any subset $S\subseteq G$ with $|S|\geq 2$, either
\begin{enumerate}[(1)]
\item there is $g\in G$ with $|SS^{g}|\geq|S|^{1+\tau}$, or
\item $G$ is the product of $\leq b$ conjugates of $S$.
\end{enumerate}
\end{theorem}

One can prove a version of the above using two distinct sets, provided that one of them is normal; the growth in this case is exceptionally fast. The result below is \cite{DMP24}.

\begin{theorem}\label{th:dmp}
For any $\varepsilon>0$ there exists $\eta=\eta(\varepsilon)>0$ such that, for any finite simple non-abelian group $G$, any subset $A\subseteq G$, and any normal subset $B\subseteq G$, if $|A|,|B|\leq|G|^{\eta}$ then $|AB|\geq|A||B|^{1-\varepsilon}$.
\end{theorem}

\begin{remark}
The result in \cite{DMP24} for groups of Lie type is proved using both an estimate of the size of small conjugacy classes of $G$, based on \cite[\S\S 3-4]{LS12} through \cite[\S 3]{LSS17}, and a non-abelian version of Pl\"unnecke's inequality that appeared in \cite[Thm.~4.1]{GPSS13}. One can prove a special case of Theorem~\ref{th:dmp} for groups of unbounded rank and for $A,B$ both conjugacy classes by using only the size estimates. The proof is more computationally cumbersome, but it relies on a nice structural resemblance between the size of $G$ and the size of a centralizer that may be of interest to the reader, and the case of conjugacy classes is all we need for this paper. 

Call $n_{i}(x,\lambda)$ the number of Jordan blocks $J_{i}(\lambda)$ of eigenvalue $\lambda\in\overline{\mathbb{F}_{q}}$ and size $i\geq 1$ appearing in the Jordan form of $x\in\mathrm{GL}_{n}(q)$. In particular, $n=\sum_{\lambda}\sum_{i}in_{i}(x,\lambda)$. Using for instance \cite[p.~3054]{FG12} for $\mathrm{SL}_{n}$ and \cite[pp.~34-36-39]{Wal63} for $\mathrm{SU}_{n},\mathrm{Sp}_{n},\mathrm{SO}_{n}^{\pm}$, one shows that $|G|\sim_{n}q^{\dim(G)}$ and $|C_{G}(x)|\sim_{n} q^{m}$ (hiding constants depending only on $n$) with
\begin{center}
\def\arraystretch{1.3}
\begin{tabular}{|c||c|c|}
\hline
$G$ & $\dim(G)$ & $m$ \\ \hline\hline
$\mathrm{SL}_{n}(q),\mathrm{SU}_{n}(q)$ & $n^{2}-1$ & $R-1$ \\ \hline
$\mathrm{Sp}_{n}(q)$ & $\frac{1}{2}(n^{2}+n)$ & $\frac{1}{2}(R+S)$ \\ \hline
$\mathrm{SO}_{n}^{\pm}(q)$ & $\frac{1}{2}(n^{2}-n)$ & $\frac{1}{2}(R-S)$ \\ \hline
\end{tabular}
\end{center}
where
\begin{align*}
R & =\sum_{\lambda}\sum_{i}\left(in_{i}(x,\lambda)^{2}+2\sum_{i'>i}in_{i}(x,\lambda)n_{i'}(x,\lambda)\right), & S & =\sum_{\lambda=\pm 1}\sum_{\text{$i$ odd}}n_{i}(x,\lambda).
\end{align*}
Then, for any $x$ having a small class, there is a special eigenvalue $\lambda_{x}$ for which $n_{1}(x,\lambda_{x})\geq\frac{1}{2}n$, and for any two such $x_{1},x_{2}$ there is $y\in\mathrm{Cl}(x_{1})\mathrm{Cl}(x_{2})$ with $\lambda_{y}=\lambda_{x_{1}}\lambda_{x_{2}}$ and
\begin{align*}
n_{1}(y,\lambda_{y}) & =n_{1}(x_{1},\lambda_{x_{1}})+n_{1}(x_{2},\lambda_{x_{2}})-n, & & \\
n_{i}(y,\lambda) & =n_{i}(x_{1},\lambda\lambda_{x_{2}}^{-1})+n_{i}(x_{2},\lambda\lambda_{x_{1}}^{-1}) & & \text{for $(\lambda,i)\neq(\lambda_{y},1)$.}
\end{align*}
Putting appropriately together all the estimates, the conclusion is that $|\mathrm{Cl}(y)|\gtrsim_{n}|\mathrm{Cl}(x_{1})||\mathrm{Cl}(x_{2})|^{1-\varepsilon}$ for $|\mathrm{Cl}(x_{1})|,|\mathrm{Cl}(x_{2})|\leq|G|^{\eta}$. The dependence is also explicit: taking $\eta=\frac{\varepsilon}{100}$ is more than enough.
\end{remark}

\section{The case $|X|\geq|C|^{\delta}$}\label{se:xcldelta}

We tackle separately the case of a subset $X$ of a conjugacy class $C$ for which $|X|\geq|C|^{\delta}$. This situation is particularly favourable, with a resulting growth compatible even with Conjecture~\ref{co:conjlns} itself. We shall use the conclusion of this section in the main proof, imposing quite arbitrarily $\delta=\frac{1}{4}$.

The idea is quite simple. Since we know that by Theorem~\ref{th:dmp} we can achieve fast growth using a normal subset, such as the class $C$, we only need to show that our large set $X$ inherits this behaviour from $C$. To do so, we use elementary combinatorial lemmas to prove that there is a rather sparse cover of $C$ by conjugates of $X$; then, since the cover is sparse, at least one of the conjugates must grow fast as well.

We start by proving that we can take any two sets in $C$ and separate them enough by conjugation.

\begin{lemma}\label{le:piececl}
Let $G$ be a finite group, and let $C\subseteq G$ be a conjugacy class. Then, for any two proper subsets $S,T\subsetneq C$, there is some $g\in G$ with $|S^{g}\cap T|\leq\min\left\{|S|-1,|S|\frac{|T|}{|C|}\right\}$.
\end{lemma}

\begin{proof}
Since $T$ is properly contained in $C$, which is the orbit of a point under the action of conjugation, there is some $g\in G$ with $S^{g}\subsetneq T$, meaning that $|S^{g}\cap T|\leq|S|-1$.

Assume now that $|S^{g}\cap T|>|S|\frac{|T|}{|C|}$ for all $g\in G$. For every $x\in C$, the number of $g$ such that $x\in S^{g}$ is $\frac{|G||S|}{|C|}$. Therefore
\begin{equation*}
\frac{|G||S|}{|C|}|T|=|\{(x,g)\in T\times G|x\in S^{g}\}|=\sum_{g\in G}|S^{g}\cap T|>|S|\frac{|T|}{|C|}|G|,
\end{equation*}
which is absurd. Thus, there must be some $g$ for which $|S^{g}\cap T|\leq|S|\frac{|T|}{|C|}$.
\end{proof}

Using the separation lemma above, we find a cover of $C$ made of conjugate sets. The cover is sparse in the sense that it is only a logarithmic factor away from the best possible cover, i.e.\ a partition.

\begin{lemma}\label{le:covercl}
Let $G$ be a finite group, and let $C\subseteq G$ be a conjugacy class. Then, for any subset $S\subseteq C$ with $|S|\geq 2$, there is some set $H\subseteq G$ with $|H|\leq\frac{6|C|\log|S|}{|S|}$ such that $\bigcup_{h\in H}S^{h}=C$.
\end{lemma}

\begin{proof}
We are going to build $H$ one element at a time: let $H_{i}$ be the set we have reached at the $i$-th step, and set $T_{i}=\bigcup_{h\in H_{i}}S^{h}$. Start with $H_{0}=T_{0}=\emptyset$; we stop when $T_{i}=C$.

At each step (assuming that $T_{i}\subsetneq C$), let $k$ be the unique positive integer such that $\left(1-\frac{1}{2^{k-1}}\right)|C|\leq|T_{i}|<\left(1-\frac{1}{2^{k}}\right)|C|$. By Lemma~\ref{le:piececl}, there is some $h_{i+1}$ with $|S^{h_{i+1}}\setminus T_{i}|\geq\frac{|S|}{2^{k}}$. Thus, if $H_{i+1}=H_{i}\cup\{h_{i+1}\}$, we get $|T_{i+1}|\geq T_{i}+\frac{|S|}{2^{k}}$. This means that, if $|T_{i}|\geq\left(1-\frac{1}{2^{k-1}}\right)|C|$, then $|T_{i+t}|\geq\left(1-\frac{1}{2^{k}}\right)|C|$ for $t=\left\lceil\frac{|C|}{|S|}\right\rceil$.

Hence, we have $|T_{a}|\geq|C|-\frac{|C|}{|S|}$ for some $a\leq\left\lceil\frac{|C|}{|S|}\right\rceil\cdot\lceil\log_{2}|S|\rceil$. By Lemma~\ref{le:piececl} again, for every $T_{i}\subsetneq C$ there is some $h_{i+1}$ with $|S^{h_{i+1}}\setminus T_{i}|\geq 1$. Thus, $T_{a+b}=C$ for some $b\leq\left\lceil\frac{|C|}{|S|}\right\rceil$.

We have reached our goal, and we set $H=H_{a+b}$. Since $|H|=a+b$, in the end
\begin{equation*}
|H|\leq \left\lceil\frac{|C|}{|S|}\right\rceil\cdot(\lceil\log_{2}|S|\rceil+1)\leq\frac{6|C|\log|S|}{|S|}
\end{equation*}
as we wanted.
\end{proof}

\begin{remark}
As pointed out by S.~Skresanov (personal communication), the nontrivial part of Lemma~\ref{le:piececl} is already in \cite[Lemma 2]{DeV23}, which in turn is a variation of old counting arguments by Babai \cite{Bab80c}. He also remarked that a much shorter probabilistic argument already gives $|H|\leq\left\lfloor\frac{|C|\log|C|}{|S|}\right\rfloor+1$ in Lemma~\ref{le:covercl}: for a fixed $x$ and random $h_{1},\ldots,h_{t}$ the probability that $x$ does not lie in $S^{h_{1}},\ldots,S^{h_{t}}$ is bounded by $e^{-t|S|/|C|}$, and it suffices to take $t>\frac{|C|\log|C|}{|S|}$ to have $|C|e^{-t|S|/|C|}<1$. The resulting bound for $|H|$ is weaker than in Lemma~\ref{le:covercl}, but not enough to affect our future arguments.
\end{remark}

Combining the cover lemma above with the growth result of Theorem~\ref{th:dmp}, we obtain the result we want for $|X|\geq|C|^{\delta}$.

\begin{proposition}\label{pr:xcldelta}
Fix $0<\delta\leq 1$ and $\eta=\eta\left(\frac{\delta}{10}\right)$, where the function $\eta$ is as in Theorem~\ref{th:dmp}. Let $G$ be a finite simple group of Lie type of rank $r$ over $\mathbb{F}_{q}$, let $C\subseteq G$ be a nontrivial conjugacy class with $q^{10r/\delta}\leq|C|\leq|G|^{\eta}$, and let $X\subseteq C$ with $|X|\geq|C|^{\delta}$. Then $G$ is the product of at most $N\frac{\log|G|}{\log|X|}$ conjugates of $X$, where $N$ depends only on $\delta$.
\end{proposition}

\begin{proof}
If $r\leq 100$ the result follows from Theorem~\ref{th:boundr}, so assume that $r>100$. We construct a sequence of sets $Y_{i}$ and conjugacy classes $C_{i}$ such that
\begin{enumerate}[(1)]
\item\label{pr:xcldelta-ycl} $Y_{i}\subseteq C_{i}$,
\item\label{pr:xcldelta-yprod} $Y_{i}$ is contained in the product of $i$ conjugates of $X$,
\item\label{pr:xcldelta-ygrow} $|Y_{i}|\geq|X|^{\frac{1}{3}i}$, and
\item\label{pr:xcldelta-clgrow} $|C_{i}|\leq|C|^{i}$.
\end{enumerate}
We are going to show that such a sequence exists, and that it does not terminate before we reach a step $k$ for which $|Y_{k}|>|G|^{\frac{1}{3}\delta\eta}$.

At the beginning, $Y_{1}=X$ and $C_{1}=C$. Now assume that the properties above hold up to the $i$-th step, and that $|Y_{i}|\leq|G|^{\frac{1}{3}\delta\eta}$: we prove that we can perform the $(i+1)$-th step. First of all, by \eqref{pr:xcldelta-ygrow} and \eqref{pr:xcldelta-clgrow},
\begin{equation*}
|G|^{\frac{1}{3}\delta\eta}\geq|Y_{i}|\geq|X|^{\frac{1}{3}i}\geq|C|^{\frac{1}{3}\delta i}\geq|C_{i}|^{\frac{1}{3}\delta},
\end{equation*}
so $|C|,|C_{i}|\leq|G|^{\eta}$. The product $C_{i}C$ is a normal set for which $|C_{i}C|\geq|C_{i}||C|^{1-\frac{\delta}{10}}$ by Theorem~\ref{th:dmp}, so using Proposition~\ref{pr:fewclass} there must be a class $C_{i+1}\subseteq C_{i}C$ for which
\begin{equation*}
|C_{i+1}|\geq\frac{|C_{i}||C|^{1-\frac{\delta}{10}}}{q^{4r}}\geq|C_{i}||C|^{1-\frac{\delta}{2}}.
\end{equation*}
Since $C_{i+1}\subseteq C_{i}C$, \eqref{pr:xcldelta-clgrow} is trivially guaranteed for $i+1$ as well. By Lemma~\ref{le:covercl} there are two sets $H,K$ such that
\begin{align*}
|H| & \leq\frac{6|C|\log|X|}{|X|}, & \bigcup_{g_{1}\in H}X^{g_{1}} & =C, \\
|K| & \leq\frac{6|C_{i}|\log|Y_{i}|}{|Y_{i}|}, & \bigcup_{g_{2}\in K}Y_{i}^{g_{2}} & =C_{i},
\end{align*}
therefore
\begin{equation*}
\bigcup_{g_{1}\in H}\bigcup_{g_{2}\in K}(X^{g_{1}}Y_{i}^{g_{2}}\cap C_{i+1})=CC_{i}\cap C_{i+1}=C_{i+1}.
\end{equation*}
So there must be one choice of $(g_{1},g_{2})\in H\times K$ such that
\begin{align*}
|X^{g_{1}}Y_{i}^{g_{2}}\cap C_{i+1}| & \geq\frac{|C_{i+1}|}{|H||K|}\geq\frac{|C|^{1-\frac{\delta}{2}}|C_{i}||X||Y_{i}|}{36|C||C_{i}|\log|X|\log|Y_{i}|}=\frac{|X||Y_{i}|}{36|C|^{\frac{\delta}{2}}\log|X|\log|Y_{i}|} \\
& \geq\frac{|X|^{\frac{1}{2}}|Y_{i}|}{36\log|X|\log|Y_{i}|}\geq\frac{|X|^{\frac{1}{2}}|Y_{i}|}{36(\log|G|)^{2}}\geq|X|^{\frac{1}{3}}|Y_{i}|,
\end{align*}
where the last inequality uses Proposition~\ref{pr:qnclass}, $|X|\geq|C|^{\delta}\geq q^{10r}$, and $r>100$. Taking $Y_{i+1}=X^{g_{1}}Y_{i}^{g_{2}}\cap C_{i+1}$, we have \eqref{pr:xcldelta-ycl} and \eqref{pr:xcldelta-yprod} for $i+1$ by definition, and \eqref{pr:xcldelta-ygrow} by what we said above. The induction is concluded.

At the final step $k$, we have finally a set $Y_{k}$ contained in the product of $k$ conjugates of $X$ and with $|Y_{k}|>|G|^{\frac{1}{3}\delta\eta}$. Using Theorem~\ref{th:ptc} repeatedly, $G$ is the product of at most $N_{1}$ conjugates of $Y_{k}$ with
\begin{equation*}
N_{1}=2^{\left\lceil\frac{\log(3/\delta\eta)}{\log(1+\tau)}\right\rceil}b.
\end{equation*}
In turn, since $|Y_{k}|\geq|X|^{\frac{1}{3}k}$, $Y_{k}$ is the product of at most $3\frac{\log|Y_{k}|}{\log|X|}$ conjugates of $X$. The result then holds with $N=3N_{1}$.
\end{proof}

\section{The main theorem}\label{se:main}

We move to the proof of the main theorem.

The idea is essentially the following (we simplify the description avoiding technicalities, so as to capture the spirit of the proof at the expense of precision). Thanks to \S\ref{se:xcldelta}, we already know what happens when our set occupies a large part of a conjugacy class, so we need to deal with the converse case. Whenever we have a set $S$ whose intersection with every class is small, there is a conjugate $S^{g}$ whose intersection with $S$ is empty. Then growth is achieved by standard tricks in map counting, similar to \cite[Lemma 2.2]{Hel19b} for instance. We then pass from $S$ to $SS^{g}$, and repeat the process, checking whether the new set has large intersection with a class or not and so on.

It is at this point that we fail to reach Conjecture~\ref{co:conjlns} and we have to settle for the exponent $1+\varepsilon$ in Theorem~\ref{th:main}. The dichotomy between the case of \S\ref{se:xcldelta} and its converse relies on the fact that the set and the class are not too far apart in size. To satisfy this condition at every step, we need to get rid of some of the elements of each new $S$, in order to avoid that the class with large intersection is too small. Thus, in the statement of Lemma~\ref{le:xclsmall} there is a discrepancy between the fact that $X'$ is contained in the product of two conjugates of $X$ and the fact that its size is $|X|^{2-\varepsilon}$ in the worst case; this discrepancy comes from the denominator $|T|+1$ in Lemma~\ref{le:cutsmall}, and in turn it leads in the end to the exponent $1+\varepsilon$.

We begin with the technical lemma that removes undesirable elements.

\begin{lemma}\label{le:cutsmall}
Let $G$ be a finite group, and let $S,T\subseteq G$ with $e\notin T=T^{-1}$. Then there is a subset $U\subseteq S$ with $|U|\geq\frac{|S|}{|T|+1}$ and $U^{-1}U\cap T=\emptyset$.
\end{lemma}

\begin{proof}
We construct two sequences of sets $U_{i},V_{i}\subseteq S$ satisfying
\begin{align}
U_{i}\cap V_{i} & =\emptyset, \label{eq:disjpro0} \\
U_{i}^{-1}U_{i}\cap T & =\emptyset, \label{eq:disjpro1} \\
U_{i}^{-1}V_{i}\cap T & =\emptyset. \label{eq:disjpro2}
\end{align}
Set $U_{0}=\emptyset$ and $V_{0}=S$, which satisfy trivially the properties above. At the $i$-th step, if $V_{i-1}$ is nonempty, we take a random $x_{i}\in V_{i-1}$ and define
\begin{align*}
U_{i} & =U_{i-1}\cup\{x_{i}\}, & V_{i} & =V_{i-1}\setminus(\{x_{i}\}\cup x_{i}T).
\end{align*}
The condition \eqref{eq:disjpro0} for $i-1$ and the construction imply \eqref{eq:disjpro0} for $i$. The conditions \eqref{eq:disjpro1}-\eqref{eq:disjpro2} for $i-1$ and $e\notin T=T^{-1}$ imply \eqref{eq:disjpro1} for $i$. The condition \eqref{eq:disjpro2} for $i-1$ and the construction imply \eqref{eq:disjpro2} for $i$.

Thus, the sequence can be built until we reach $V_{k}=\emptyset$. We always reach this stage, because $|V_{i}|<|V_{i-1}|$ for all $i$. Moreover, $|U_{i}|=|U_{i-1}|+1$ and $|V_{i}|\geq|V_{i-1}|-|T|-1$. Therefore, $U=U_{k}\subseteq S$ at the final step satisfies $|U|\geq\frac{|S|}{|T|+1}$, and \eqref{eq:disjpro1} means that $U^{-1}U\cap T=\emptyset$.
\end{proof}

For a fixed group $G$ and a fixed $0<\varepsilon<1$, define the following two properties for any subset $X\subseteq G$:
\begin{align*}
P_{1}(X): & \text{ every conjugacy class $C\neq\{e\}$ with $|C|<|X|^{\frac{1}{3}\varepsilon}$ has $X^{-1}X\cap C=\emptyset$,} \\
P_{2}(X): & \text{ every conjugacy class $C\neq\{e\}$ with $|C|\geq|X|^{\frac{1}{3}\varepsilon}$ has $|X^{-1}X\cap C|<|C|^{\frac{1}{4}}$} \\
& \text{ and $|XX^{-1}\cap C|<|C|^{\frac{1}{4}}$.}
\end{align*}
In \S\ref{se:xcldelta} we dealt with $X$ satisfying $P_{1}(X)$ and not $P_{2}(X)$. Now we prove growth for $X$ satisfying both $P_{1}(X)$ and $P_{2}(X)$.

\begin{lemma}\label{le:xclsmall}
Let $0<\varepsilon<1$, and let $G$ be a finite simple group of Lie type. For any $X$ with $|X|\geq q^{100r/\varepsilon}$ such that $P_{1}(X),P_{2}(X)$ hold, there are some $g\in G$ and some $X'\subseteq XX^{g}$ with $|X'|\geq|X|^{2-\varepsilon}$ such that $P_{1}(X')$ holds.
\end{lemma}

\begin{proof}
We start by arguing similarly to Lemma~\ref{le:piececl}. Let $C$ be any nontrivial conjugacy class with $|C|\geq|X|^{\frac{1}{3}\varepsilon}$, and let $Z_{1}=X^{-1}X\cap C$ and $Z_{2}=XX^{-1}\cap C$. Assume that there are $\geq\frac{|G|}{2q^{4r}}$ values of $g\in G$ such that $Z_{1}\cap Z_{2}^{g}\neq\emptyset$. Then
\begin{equation*}
\frac{|G||Z_{2}|}{|C|}|Z_{1}|=|\{(x,g)\in Z_{1}\times G|x\in Z_{2}^{g}\}|=\sum_{g\in G}|Z_{1}\cap Z_{2}^{g}|\geq\frac{|G|}{2q^{4r}},
\end{equation*}
implying that $|C|\leq 2q^{4r}|Z_{1}||Z_{2}|$. However, $2q^{4r}<|X|^{\frac{1}{6}\varepsilon}\leq|C|^{\frac{1}{2}}$, so we obtain $\max\{|Z_{1}|,|Z_{2}|\}\geq|C|^{\frac{1}{4}}$ contradicting $P_{2}(X)$. Hence, there are $>\left(1-\frac{1}{2q^{4r}}\right)|G|$ values of $g\in G$ such that $Z_{1}\cap Z_{2}^{g}=\emptyset$.

By Proposition~\ref{pr:fewclass} there are at most $q^{4r}$ conjugacy classes in $G$, and for the nontrivial classes with $|C|<|X|^{\frac{1}{3}\varepsilon}$ we know that $X^{-1}X\cap C=\emptyset$ by $P_{1}(X)$. Thus, there must exist some $g\in G$ for which $(X^{-1}X\cap C)\cap(XX^{-1}\cap C)^{g}=\emptyset$ for all nontrivial conjugacy classes $C$ at the same time. In other words, we must have $X^{-1}X\cap(XX^{-1})^{g}=\{e\}$ for some $g$, which implies $|XX^{g}|=|X|^{2}$.

Finally, let $T$ be the union of all nontrivial classes of size $<|XX^{g}|^{\frac{1}{3}\varepsilon}$; then $e\notin T=T^{-1}$ and $|T|\leq q^{4r}|X|^{\frac{2}{3}\varepsilon}<\frac{1}{2}|X|^{\varepsilon}$. Applying Lemma~\ref{le:cutsmall} to $S=XX^{g}$, the resulting subset $U$ satisfies $|U|\geq|X|^{-\varepsilon}|XX^{g}|=|X|^{2-\varepsilon}$ and $U^{-1}U\cap T=\emptyset$, which implies a fortiori $P_{1}(U)$. Take $X'=U$.
\end{proof}

The formulation of Lemma~\ref{le:xclsmall} now makes the induction process evident. At this point we have all the tools to prove the main theorem.

\begin{proof}[Proof of Theorem~\ref{th:main}]
We make the set $A$ grow in stages.

If $|A|<q^{\frac{1}{4}r}$, then by Proposition~\ref{pr:smallgrow} there is some $m$ and there are $g_{1},\ldots,g_{m}$ such that
\begin{align*}
A_{1} & =A^{g_{1}}A^{g_{2}}\ldots A^{g_{m}}, & |A_{1}|=|A|^{m}\geq\frac{q^{\frac{1}{2}r}}{|A|}\geq q^{\frac{1}{4}r}.
\end{align*}
In this case, set $N_{1}=m=\frac{\log|A_{1}|}{\log|A|}$. If on the other hand $|A|\geq q^{\frac{1}{4}r}$, set $A_{1}=A$ and $N_{1}=1$. In either situation, $A_{1}$ is a product of $N_{1}$ conjugates of $A$, and $|A_{1}|\geq q^{\frac{1}{4}r}$.

If $|A_{1}|<q^{\frac{240}{\varepsilon}r}$, set $A_{2,0}=A_{1}$, and for every $A_{2,i-1}$ we know by Theorem~\ref{th:ptc} that either there is some $g$ for which we get
\begin{align*}
A_{2,i} & =A_{2,i-1}A_{2,i-1}^{g}, & |A_{2,i-1}|^{1+\tau}\leq|A_{2,i}|\leq|A_{2,i-1}|^{2}
\end{align*}
or there are $g_{1},\ldots,g_{b}$ such that
\begin{equation*}
A_{2,i-1}^{g_{1}}\ldots A_{2,i-1}^{g_{b}}=G.
\end{equation*}
We stop at the smallest $k$ for which
\begin{align*}
A_{2,k} & =G & & \text{or} & q^{\frac{240}{\varepsilon}r}\leq|A_{2,k}|\leq q^{\frac{480}{\varepsilon}r}.
\end{align*}
Setting $A_{2}=A_{2,k}$, we have that $A_{2}$ is a product of at most $N_{2}$ conjugates of $A_{1}$ with
\begin{align*}
N_{2} & =2^{k}b, & k & \leq\frac{\log\log|A_{2}|-\log\log|A_{1}|}{\log(1+\tau)}\leq\frac{\log\left(\frac{1920}{\varepsilon}\right)}{\log(1+\tau)},
\end{align*}
If we already have $|A_{1}|\geq q^{\frac{240}{\varepsilon}r}$, we set $A_{2}=A_{1}$ and $N_{2}=1$; the lower bound on $|A_{2}|$ and the facts concerning $N_{2}$ that we had before hold trivially in this case too.

Let $T$ be the union of all nontrivial classes of size $<|A_{2}|^{\frac{1}{3}\varepsilon}$; then $e\notin T=T^{-1}$ and $|T|\leq q^{4r}|A_{2}|^{\frac{1}{3}\varepsilon}<\frac{1}{2}|A_{2}|^{\varepsilon}$ by Proposition~\ref{pr:fewclass}. Applying Lemma~\ref{le:cutsmall} to $S=A_{2}$, the resulting subset $U=A_{3}$ satisfies $|A_{3}|\geq|A_{2}|^{1-\varepsilon}$ and $A_{3}^{-1}A_{3}\cap T=\emptyset$, which implies a fortiori $P_{1}(A_{3})$. Note also that $|A_{3}|\geq|A_{2}|^{1-\varepsilon}\geq q^{\frac{120}{\varepsilon}r}$.

We construct a sequence of sets $A_{4,i}$, starting with $A_{4,0}=A_{3}$: we apply Lemma~\ref{le:xclsmall} as long as $P_{2}(A_{4,i})$ holds, and call $A_{4,i+1}=X'$ as in the conclusion of that result. We stop at the smallest $k$ for which $P_{2}(A_{4,k})$ does not hold (it may be that $k=0$ already). At the final step, we have reached a set $A_{4}=A_{4,k}$ contained in the product of $N_{4}$ conjugates of $A_{3}$ (and of $A_{2}$ as well), with the following properties:
\begin{align*}
|A_{4}| & \geq|X_{3}|^{(2-\varepsilon)^{k}} & & \Longrightarrow & k & \leq\frac{\log\log|A_{4}|-\log\log|A_{3}|}{\log(2-\varepsilon)}, \\
N_{4} & =2^{k} & & \Longrightarrow & N_{4} & \leq\left(\frac{\log|A_{4}|}{\log|A_{3}|}\right)^{\frac{\log 2}{\log(2-\varepsilon)}}<\left(\frac{\log|A_{4}|}{\log|A_{3}|}\right)^{1+\varepsilon} \\
& & & & & \leq \left(\frac{\log|A_{4}|}{(1-\varepsilon)\log|A_{2}|}\right)^{1+\varepsilon}\leq 2\left(\frac{\log|A_{4}|}{\log|A_{2}|}\right)^{1+\varepsilon}.
\end{align*}

At this point we have $A_{4}$ for which $P_{1}(A_{4})$ holds and $P_{2}(A_{4})$ does not hold. In particular, there is some nontrivial class $C$ for which $|C|\geq|A_{4}|^{\frac{1}{3}\varepsilon}$ and either $|A_{4}^{-1}A_{4}\cap C|\geq|C|^{\frac{1}{4}}$ or $|A_{4}A_{4}^{-1}\cap C|\geq|C|^{\frac{1}{4}}$. Since $|A_{4}|\geq|A_{3}|\geq q^{\frac{120}{\varepsilon}r}$, it must be that $|C|\geq q^{40r}$. Let $A_{5}$ be the set of larger size between $A_{4}^{-1}A_{4}\cap C$ and $A_{4}A_{4}^{-1}\cap C$: in either case, by definition $A_{5}$ is contained in the product of $N_{5}=2$ conjugates of either $A_{4}$ or $A_{4}^{-1}$.

We have two cases. Let $\eta=\eta\left(\frac{1}{40}\right)$ as in Theorem~\ref{th:dmp}, and suppose that $|C|\leq|G|^{\eta}$. Then, we can use Proposition~\ref{pr:xcldelta} with $\delta=\frac{1}{4}$ and $X=A_{5}$; hence, $G$ is the product of $N_{6,1}$ conjugates of $A_{5}$ for some absolute constant $N_{6,1}$. Suppose instead that $|C|>|G|^{\eta}$, which means that $|A_{5}|\geq|G|^{\frac{\eta}{4}}$. Then, we can apply Theorem~\ref{th:ptc} repeatedly as we did above, showing that $G$ is the product of $N_{6,2}$ conjugates of $A_{5}$ for
\begin{align*}
N_{6,2} & =2^{k'}b, & k' & \leq\frac{\log\log|G|-\log\log|A_{5}|}{\log(1+\tau)}\leq\frac{\log\left(\frac{4}{\eta}\right)}{\log(1+\tau)}.
\end{align*}
Define $N_{6}=\max\{N_{6,1},N_{6,2}\}$.

Putting all the stages together, $G$ is the product of $N$ conjugates of either $A$ or $A^{-1}$ with
\begin{align}
N & =N_{1}N_{2}N_{4}N_{5}N_{6} \nonumber \\
& \leq\frac{\log|A_{1}|}{\log|A|}\cdot N_{2}\cdot 2\left(\frac{\log|A_{4}|}{\log|A_{2}|}\right)^{1+\varepsilon}\cdot 2\cdot N_{6} \label{eq:finale} \\
& \leq 4N_{2}N_{6}\cdot\left(\frac{\log|G|}{\log|A|}\right)^{1+\varepsilon}. \nonumber
\end{align}
Since $N_{6}$ is an absolute constant and $N_{2}$ depends only on $\varepsilon$, the result follows.
\end{proof}

\begin{remark}\label{re:pyber}
As remarked by L.~Pyber (personal communication), the number of conjugates can also be bounded as
\begin{equation*}
N_{\varepsilon}r^{\varepsilon}\cdot\frac{\log|G|}{\log|A|}
\end{equation*}
where $r$ is the rank of $G$. In fact, observe that in our proof $|A_{2}|\geq q^{r}$ and that $|A_{4}|\leq|G|\leq q^{8r^{2}}$ by Proposition~\ref{pr:qnclass}, so inside \eqref{eq:finale} one can use the bound $\frac{\log|A_{4}|}{\log|A_{2}|}\leq 8r$.

This provides a more direct improvement to \eqref{eq:boundgpss}. In the proof of \cite[Thm.~1.3]{GPSS13}, the resulting function is of the form $N(r)=O(r^{e^{1/\varepsilon(r)}})$ where $\varepsilon(r)$ comes from using some version of the product theorem for Babai's conjecture. Depending on which version we use, one may obtain $1/\varepsilon(r)$ being: a non-explicit function \cite{BGT11}, a tower of exponential of height depending on $r$ \cite{PS16}, or a polynomial in $r$ for untwisted classical groups \cite{BDH21}. Furthermore, it is known by explicit counterexamples \cite[Ex.~77]{PS16} that passing through such theorems one cannot do better than $1/\varepsilon(r)$ linear in $r$. If we used Theorem~\ref{th:ptc} instead of product theorems as above, we would still obtain $N(r)=O(r^{N})$ for some constant $N$. In all cases, $N(r)=O_{\varepsilon}(r^{\varepsilon})$ as we have here is a better bound.
\end{remark}

\section*{Acknowledgements}

The author thanks several members of the R\'enyi Institute (A.~Mar\'oti, L.~Sabatini, S.~Skresanov, D.~Szab\'o) for discussions about the problem and its techniques, which helped improve the structure of the proof. The author was funded by a Young Researcher Fellowship from the R\'enyi Institute.

\bibliography{Bibliography}

\begin{thebibliography}{GLPS24}

\bibitem[Bab80]{Bab80c}
L.~Babai.
\newblock On the complexity of canonical labeling of strongly regular graphs.
\newblock {\em SIAM J. Comput.}, 9(1):212--216, 1980.

\bibitem[BDH21]{BDH21}
J.~Bajpai, D.~Dona, and H.~A. Helfgott.
\newblock Growth estimates and diameter bounds for classical {C}hevalley
  groups.
\newblock \texttt{arXiv:2110.02942}, 2021.

\bibitem[BGT11]{BGT11}
E.~Breuillard, B.~Green, and T.~Tao.
\newblock Approximate subgroups of linear groups.
\newblock {\em Geom. Funct. Anal.}, 21(4):774--819, 2011.

\bibitem[BS88]{BS88}
L.~Babai and \'A. Seress.
\newblock On the diameter of {C}ayley graphs of the symmetric group.
\newblock {\em J. Combin. Theory Ser. A}, 49(1):175--179, 1988.

\bibitem[DeV23]{DeV23}
M.~DeVos.
\newblock Longer cycles in vertex transitive graphs.
\newblock \texttt{arXiv:2302.04255}, 2023.

\bibitem[DMP24]{DMP24}
D.~Dona, A.~Mar\'oti, and L.~Pyber.
\newblock Growth of products of subsets in finite simple groups.
\newblock {\em Bull. Lond. Math. Soc.}, 56(8):2704--2710, 2024.

\bibitem[FG12]{FG12}
J.~Fulman and R.~Guralnick.
\newblock Bounds on the number and sizes of conjugacy classes in finite
  {C}hevalley groups with applications to derangements.
\newblock {\em Trans. Amer. Math. Soc.}, 364(6):3023--3070, 2012.

\bibitem[GLPS24]{GLPS24}
N.~Gill, N.~Lifshitz, L.~Pyber, and E.~Szab\'o.
\newblock Initiating the proof of the {L}iebeck-{N}ikolov-{S}halev conjecture.
\newblock \texttt{arXiv:2408.07800}, 2024.

\bibitem[GPSS13]{GPSS13}
N.~Gill, L.~Pyber, I.~Short, and E.~Szab\'o.
\newblock On the product decomposition conjecture for finite simple groups.
\newblock {\em Groups Geom. Dyn.}, 7(4):867--882, 2013.

\bibitem[Hel08]{Hel08}
H.~A. Helfgott.
\newblock Growth and generation in $\mathrm{SL}_{2}(\mathbb{Z}/p\mathbb{Z})$.
\newblock {\em Ann. of Math. (2)}, 167:601--623, 2008.

\bibitem[Hel19]{Hel19b}
H.~A. Helfgott.
\newblock Growth in linear algebraic groups and permutation groups: towards a
  unified perspective.
\newblock In C.~M. Campbell, C.~W. Parker, M.~R. Quick, E.~F. Robertson, and
  C.~M. Roney-Dougal, editors, {\em Groups St Andrews 2017 in Birmingham},
  volume 455 of {\em London Mathematical Society Lecture Note Series}, pages
  300--345. Cambridge University Press, Cambridge (UK), 2019.

\bibitem[KL90]{KL90}
P.~Kleidman and M.~Liebeck.
\newblock {\em The subgroup structure of the finite classical groups}, volume
  129 of {\em London Mathematical Society Lecture Note Series}.
\newblock Cambridge University Press, Cambridge (UK), 1990.

\bibitem[Lif24]{Lif24}
N.~Lifshitz.
\newblock Completing the proof of the {L}iebeck-{N}ikolov-{S}halev conjecture.
\newblock \texttt{arXiv:2408.10127}, 2024.

\bibitem[LNS10]{LNS10}
M.~W. Liebeck, N.~Nikolov, and A.~Shalev.
\newblock A conjecture on product decompositions in simple groups.
\newblock {\em Groups Geom. Dyn.}, 4(4):799--812, 2010.

\bibitem[LNS12]{LNS12}
M.~W. Liebeck, N.~Nikolov, and A.~Shalev.
\newblock Product decompositions in finite simple groups.
\newblock {\em Bull. Lond. Math. Soc.}, 44(3):469--472, 2012.

\bibitem[LP97]{LP97}
M.~W. Liebeck and L.~Pyber.
\newblock Upper bounds for the number of conjugacy classes of a finite group.
\newblock {\em J. Algebra}, 198(2):538--562, 1997.

\bibitem[LS12]{LS12}
M.~W. Liebeck and G.~M. Seitz.
\newblock {\em Unipotent and nilpotent classes in simple algebraic groups and
  {L}ie algebras}, volume 180 of {\em Mathematical Surveys and Monographs}.
\newblock American Mathematical Society, Providence (USA), 2012.

\bibitem[LSS17]{LSS17}
M.~W. Liebeck, G.~Schul, and A.~Shalev.
\newblock Rapid growth in finite simple groups.
\newblock {\em Trans. Amer. Math. Soc.}, 369(12):8765--8779, 2017.

\bibitem[LT24]{LT24}
M.~J. Larsen and P.~H. Tiep.
\newblock Uniform character bounds for finite classical groups.
\newblock {\em Ann. of Math. (2)}, 200(1):1--70, 2024.

\bibitem[PS16]{PS16}
L.~Pyber and E.~Szab\'o.
\newblock Growth in finite simple groups of {L}ie type.
\newblock {\em J. Amer. Math. Soc.}, 29(1):95--146, 2016.

\bibitem[Wal63]{Wal63}
G.~E. Wall.
\newblock On the conjugacy classes in the unitary, symplectic and orthogonal
  groups.
\newblock {\em J. Aust. Math. Soc.}, 3(1):1--62, 1963.

\end{thebibliography}
\bibliographystyle{alpha}

\end{document}